\documentclass{amsart}
\usepackage{amsmath}
  \usepackage{paralist}
    \usepackage{color}
  \usepackage{graphics} 
  \usepackage{epsfig} 
\usepackage{graphicx}  \usepackage{epstopdf}

\usepackage[english]{babel}
\usepackage{amssymb}

\usepackage{xcolor,colortbl}

\newcommand\deq{\mathrel{\stackrel{\makebox[0pt]{\mbox{\normalfont\tiny def}}}{=}}}
\newcommand{\veps}{\varepsilon}
\usepackage{braket}

  \DeclareMathOperator{\Ker}{Ker}

\def\F{\mathbf F}

\def\cA{\mathcal A}
\def\cB{\mathcal B}

\def\cL{\mathcal L}

\def\cU{\mathcal U}

\def\Ker{\mbox{\rm Ker}}

\def\Sym{\mbox{\rm Sym}}
\def\dim{\mbox{\rm dim}}

\def\FF{{\mathbb F}}

\def\f2{{\mathbb F}_{2}}

\newcommand{\AGL}{\mbox{\rm AGL}}

\newcommand{\GL}{\mbox{\rm GL}}

\newcommand{\g}{\gamma}
\newcommand{\gd}{\delta}

\newcommand{\gl}{\lambda}

\newcommand{\gr}{\rho}
\newcommand{\gs}{\sigma}

\renewcommand{\phi}{\varphi}

\DeclareMathOperator{\sym}{Sym}

\def\F2{\mathbb{F}_{\hspace{-0.7mm}2}}

\newtheorem{theorem}{Theorem}[section]

\newtheorem{lemma}[theorem]{Lemma}
\newtheorem{proposition}{Proposition}

\theoremstyle{definition}
\newtheorem{definition}[theorem]{Definition}
\newtheorem{remark}{Remark}

\title[] 
      {Some group-theoretical results on Feistel Networks in a long-key scenario}

\author[R. Aragona, M. Calderini, and R. Civino]{}


 \email{riccardo.aragona@univaq.it}
 \email{marco.calderini@uib.no }
 \email{roberto.civino@univaq.it}

\begin{document}
\maketitle

\centerline{\scshape Riccardo Aragona}
{\footnotesize
 \centerline{DISIM, University of L'Aquila }
   \centerline{ Via Vetoio, 67100 Coppito (AQ), Italy}
}
\medskip

   \centerline{\scshape Marco Calderini}
 {\footnotesize   \centerline{Department of Informatics, University of Bergen}
   \centerline{ Postboks 7803, N-5020 Bergen, Norway}
   }
   \medskip
   
\centerline{\scshape Roberto Civino}
{\footnotesize
 \centerline{DISIM, University of L'Aquila }
   \centerline{ Via Vetoio, 67100 Coppito (AQ), Italy}
} 

%

\bigskip


\begin{abstract}
The study of the trapdoors that can be hidden in a block cipher is and has always been a high-interest topic in symmetric cryptography. In this paper we focus on Feistel-network-like ciphers in a classical long-key scenario and we investigate some conditions which make such a construction immune to the partition-based attack introduced recently by Bannier et al.  
\end{abstract}

\section{Introduction}\label{intro}
Most modern block ciphers belong to two families of symmetric cryptosystems, i.e. Substitution-Permutation Networks (SPN) and Feistel Networks. Typically, in both cases,  each encryption function is a composition of key-dependent permutations of the plaintext space, called \emph{round functions}, designed in a such way to provide both \emph{confusion} and \emph{diffusion} (see \cite{sha49}). Confusion is provided applying  public non-linear vectorial Boolean functions, called S-boxes, whereas diffusion is obtained by means of public linear maps, called diffusion layers. The private component of the cipher, i.e. the \emph{key}, is derived from the user-provided information by means of a public procedure known as \emph{key-schedule}. When the round functions are made in such a way the confusion and diffusion layers are followed by the XOR-addition with the so-called \emph{round-key}, where the round-key is every possible vector in the message space, the cipher is a \emph{long-key cipher}.\\

Since the seventies,  many researchers have studied the relationship between some algebraic properties of the confusion / diffusion layers and some algebraic weaknesses of the corresponding ciphers, using a permutation-group-theoretical approach.
In 1975, Coppersmith and Grossman \cite{coppersmith1975generators} considered a set of permutations which can be used to define a block cipher and, by studying the permutation group that they generate, they linked some properties of this group and the security of the corresponding cipher. From this work 
 a new branch of research was born, which focuses on group-theoretical properties that can be exploited to attack encryption methods.
In \cite{kalinski}, the authors proved that if the permutation group generated by the encryption functions of a cipher is too small, then the cipher is vulnerable
to birthday-paradox attacks. In \cite{Ca15} the authors proved that if such group is  isomorphic to a subgroup of the affine group of the plaintext space, induced by a  sum different to the classical bitwise XOR, then it is possible to embed a dangerous trapdoor on it.  More relevant in \cite{CGC-cry-art-paterson1}, Paterson built a DES-like~\cite{DES} cipher whose encryption functions generate an imprimitive group 
and showed how the knowledge of this trapdoor can be turned into an efficient attack to the cipher. 
For this reason, showing that the group generated by the encryption functions of a given cipher is primitive and not of affine type became a relevant branch of research (see \cite{PriPre, GOST_ric, ACDVS,  CGC-cry-art-carantisalaImp, ONAN, Wernsdorf2, We1,We2,We3}).
Recently, in~\cite{bannier2016partition,bannier2017partition} the imprimitive attack shown by Paterson was generalized  by means of a trapdoor which consists in mapping a partition of the plaintext space into a (different) partition of the ciphertext space.
The authors also proved that only \emph{linear} partitions can propagate round-by-round in a long-key SPN. Later Calderini~\cite{calderini} has shown  which conditions ensure that linear partitions cannot propagate in a long-key SPN. \\

In this work we study some properties of the linear-partition propagation under the action of a long-key Feistel network. In particular, our aim is to prove that also in a Feistel-network-like long-key framework, if the cipher allows partition propagation, then the partitions are linear one. Moreover, we provide a partial generalisation of Calderini's result in the Feistel network case. 

\section{Preliminaries and notation}\label{sect2}
The notation and parameters which are used throughout this paper are presented in the following section.\\

Let $n\in \mathbb N$ and let us denote $V=(\mathbb F_2)^n$ the $n$-dimensional vector space over $\mathbb F_2$ equipped with the bit-wise XOR. Let us suppose $\dim(V)=n=bs$ and let us write $V= V_1\oplus V_2\oplus \ldots\oplus V_b$ where  for $1\leq j\leq b$, $\dim(V_j) = s$ and $\oplus$ represents the direct sum of vector subspaces. The subspaces $V_j$ are called \emph{bricks}. For any $I\subset \{1,...,b\}$, with $I\ne \emptyset$ and $I\ne \{1,...,b\}$, the direct sum $\bigoplus_{i\in I} V_i$ is called a {\em wall}. 
We denote by $\sym(V)$ the symmetric group acting on $V$, i.e. the group of all the permutations on $V$. Let us also denote by $\AGL(V)$ the group of all affine permutations of $V$, which is a primitive maximal subgroup of $\Sym(V)$. The translation group on $V$ is denoted by $T(V)$, i.e. 
$T(V) \deq \left\{\sigma_v \mid v \in V, \,  x \mapsto x+v \right\} < \Sym(V)$.\\

Let us now introduce block ciphers, the subject of this work.

\subsection{Block ciphers}
Let  $\mathbb M$ and  $\mathbb K$ be non-empty sets, where $| \mathbb K| \geq | \mathbb M|$.
A \emph{block cipher} $\Phi$ is a family of key-dependent permutations 
\[\{E_K \mid E_K: \mathbb M \rightarrow \mathbb M, \, K \in \mathbb K \},\]
 where $\mathbb M$ is called the message space and $\mathbb K$ the key space. The permutation $E_K$ is called the \emph{encryption function induced by the master key} $K$. The block cipher $\Phi$ is called an \emph{iterated block cipher} if there exists $r \in \mathbb N$ such that for each $K \in \mathbb K$ the encryption function $E_K$ is the composition of $r$ key-dependent \emph{round functions}, i.e. $E_K = \veps_{1,K}\,\veps_{2,K}\ldots\veps_{r,K}$. To provide efficiency, each round function is the composition of a public component provided by the designers, and a private component derived from the user-provided key by means of a public procedure known as \emph{key-schedule}. \\

\noindent In the theory of modern iterated block cipher, two frameworks are mainly considered: Substitution-Permutation Networks, typically abbreviated as SPN  (see e.g. AES \cite{daemen2002design}) and Feistel networks (see e.g. \cite{DES}). Figure \ref{fig.rounds} depicts the more general framework of SPNs, Feistel networks and their round functions; notice that  inside the round function of  a Feistel network, a function called F-function is applied to a half of the state. In both cases, the principles of confusion and diffusion suggested by Shannon \cite{sha49} are implemented by considering each round function (or respectively {F-function}) as the composition of key-induced permutation as well as non-linear confusion layers and linear diffusion layers, which are invertible in the case of SPNs and preferably (but not necessarily) invertible in the case of Feistel networks.  The following definition has been given in \cite{wave} and introduces a class of round functions for iterated block ciphers which is large enough to include the round functions of well-established SPNs and some F-functions of Feistel networks. \\

\begin{figure}
\centering
\includegraphics[scale = 0.105]{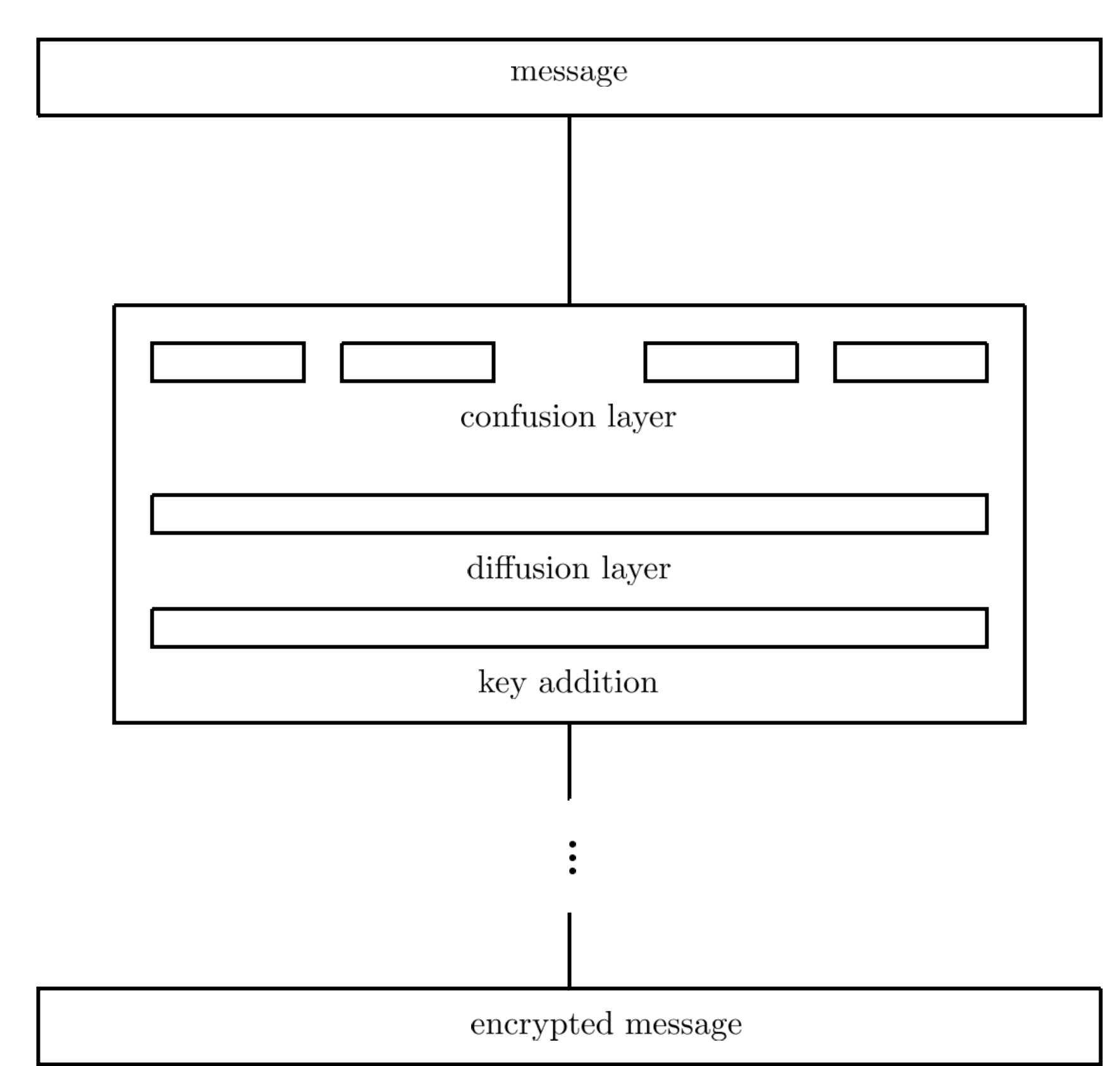}\hfil
\includegraphics[scale= 0.12]{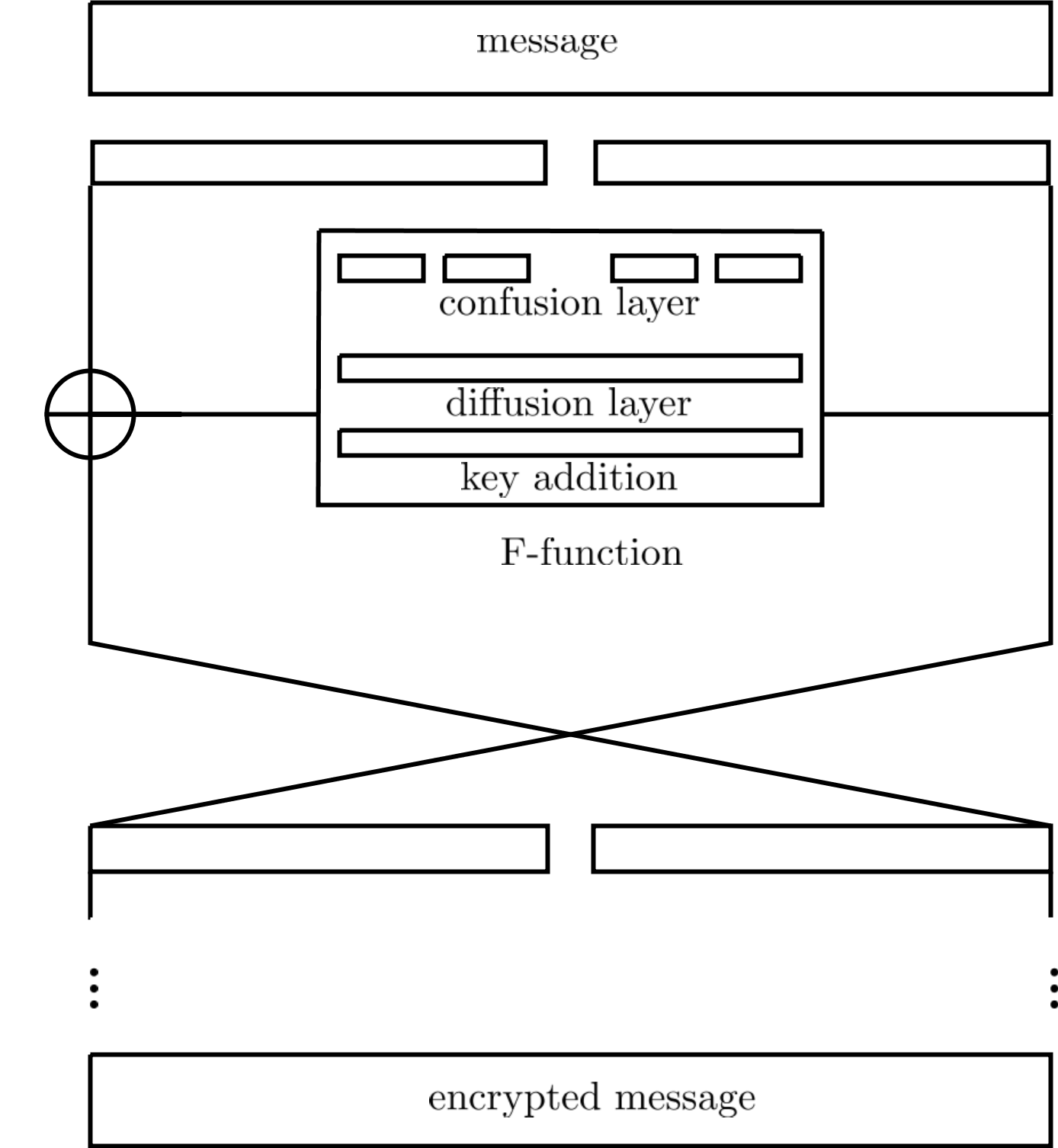}
\caption{Round function of an SPN and of a Feistel network}\label{fig.rounds}
\end{figure}

\begin{definition}
A \emph{classical round function} is a map of the type  $\veps_{k} = \gamma \lambda \sigma_{k} \in \Sym(V)$, where $k \in V$ and 
\begin{itemize}
\item $\gamma : V \rightarrow V$ is a non-linear permutation (parallel S-box) which acts in parallel way on each $V_{j}$, i.e. \[(x_1,x_2,\ldots,x_n)\gamma = \left((x_1,\ldots,x_{s})\gamma^{(1)},\ldots,(x_{s(b-1)+1},\ldots,x_{n})\gamma^{(b)}\right).\] The maps $\gamma^{(j)}  :  V_j \rightarrow V_j $ are traditionally called S-boxes;
 \item $\lambda \in \sym(V)$ is a linear map, called {\em diffusion layer};
 \item  $\sigma_{k}: V \rightarrow V, x \mapsto x+k$, called \emph{key-addition} layer,  represents the addition with the round key $k$, where $+$ is the usual bitwise \emph{XOR}.
 \end{itemize}
\end{definition}

In modern literature, terms \emph{SPN} (or the similar notion of \emph{translation-based cipher}~\cite{CGC-cry-art-carantisalaImp}) and \emph{Feistel network} may refer to a very diverse variety of ciphers. For the purposes of this paper we choose to focus only on ciphers with an XOR-based key addition. For this reason, saying SPN we refer to any cipher $\{E_K \mid K \in \mathbb K \} \subseteq \sym(\mathbb M)$ having an SPN-like structure with $\mathbb M = V$ and having classical round functions on $V$ as round functions, and saying Feistel network to any cipher $\{E_K \mid K \in \mathbb K \} \subseteq \sym(\mathbb M)$ having a Feistel-network-like structure with  $\mathbb M = V\times V$ and having classical round functions on $V$ as F-functions. In both cases, the composition $\rho_i \deq \gamma_i\lambda_i$ is called the \emph{generating function} of the $i$-th round of the cipher. Notice that usually in real-life ciphers it holds $\rho_1 = id_V$, which means that in the first round only a key addition is applied to the plaintext (\emph{whitening}).  In this setting, an $r$-round cipher is defined once the list of its generating functions $\rho_1,\ldots,\rho_r$ and its key-schedule are given.\\

Once the key $K \in \mathbb K$ to be used for the encryption has been chosen, the encryption function is obtained by composing the $r$ classical round functions induced by the corresponding round keys, which are, as previously mentioned, derived by the key-schedule. Hence, in the quite popular setting in which the round key is XORed to the state, the key-schedule is a  function 
\[
\mathcal S: \mathbb K  \rightarrow V^r
\]
such that $\mathcal S(K) \deq (k_1,\ldots,k_r)$ for any $K\in\mathbb K$, where  $\mathcal S(K)_i\deq k_i$ is the $i$-th round key derived from the user-provided key $K$ and $\veps_{i,K} = \veps_{\mathcal S(K)_i }$. \\

In the following section we recall some basic security notion for Boolean function that we will use later. 
\subsection{Security notions for Boolean functions}
The following property is the standard request for the linear component of a block cipher to spread the input bits as much as possible within the ciphertext.

\begin{definition}
A linear map $\gl\in \GL(V)$ is called a {\em proper diffusion layer} if no wall is invariant under $\gl$ and it is called  a {\em strongly
proper diffusion layer} if there are no walls $W$ and $W'$ such that $W\gl = W'$.
\end{definition}
In the remainder of this section we recall notions of non-linearity which will be useful in this work. 
Let us recall that the non-linear layer of the ciphers which will be considered throughout this work act applying  vectorial Boolean functions $\gamma^{(i)}$ to each brick of the block. Notice that we can always assume $0\gamma^{(i)} = 0$ without loss of generality, since otherwise $0\gamma^{(i)}$ can be included as part of key-addition layer of the previous round, for each round index $1 \leq i \leq b$ (see \cite[Remark 3.3]{CGC-cry-art-carantisalaImp}).

\begin{definition}
Let $f \in \Sym\left((\FF_2)^s\right)$.
Let us define
\[
\gd_f(a,b)=|\{x\in{(\FF_2)^s} \mid\, xf+(x+a)f=b\}|.
\]
The map $f$ is said  $\gd$-differentially uniform if 
\[
{\gd}=\max\limits_{\substack{a,b \\ a \ne 0}}
\gd_f(a,b).
\]
\end{definition}
\noindent It is known that  $\delta$-differentially uniform functions with small $\delta$ are ``farther'' from being linear compared to functions to with a larger differential uniformity. Notice indeed that when $f$ is linear, then $\delta=2^s$. Let us recall that $2$-differentially uniform S-boxes, which reach the lower bound of the previous definition, are called  \emph{Almost Perfect Non-linear (APN)}.
Vectorial Boolean functions used as S-boxes in block ciphers must have low uniformity to prevent differential cryptanalysis (see \cite{CGC2-cry-art-biham1991differential}) and so APN S-boxes usually represent an optimal choice in terms of resistance to differential attacks.
\medskip

We conclude this section giving another notion of non-linearity that we will use in some results of this work. 

\begin{definition}[\cite{CGC-cry-art-carantisalaImp}]
Let $1\leq \delta <s$ and $f \in \Sym((\FF_2)^s)$ such that $f(0)=0$. The function $f$ is {\it strongly $\delta$-anti-invariant} if for each  $U$ and $W$ proper and non-trivial subspaces of $(\mathbb F_2)^s$, then
\[
Uf=W \implies \dim(U)=\dim(W)<s-\delta.
\]
\end{definition}
\noindent Notice that if $1 \leq \delta< \delta' < s$ and $f$ is strongly $\delta'$-anti-invariant, then it is also strongly $\delta$-anti-invariant.

\subsection{A long-key scenario}
As mentioned in the introduction, the focus of this work is on a specific type of key-schedule, i.e. the one defined as follows:
\begin{definition}\label{def:ks}
Let $\Phi$ be an $r$-round cipher on $\mathbb M$ and let $\mathcal S: \mathbb K \rightarrow V^r$ its key-schedule. Then $\Phi$ is called a \emph{long-key cipher} if $\mathcal S (\mathbb K) = V^r$.
\end{definition}

The group generated by the encryption functions of a long-key cipher and its properties will be investigated throughout this work.
In the next section we will, in particular, study its behavior in relation to the attacks described in the following section.
\section{Group-theoretical trapdoors}
The study of groups related to block ciphers may reveal weaknesses which can be  exploited to perform algebraic attacks. In this paper, we focus on some particular group-theoretical attacks (see e.g. \cite{CGC-cry-art-paterson1,bannier2016partition}), based on undesirable properties of such permutation groups.
Notice that the study of the group generated by the encryption functions is a hard task in general, since the dependence on the key-schedule is not easily turned into algebraic conditions. The aim of this work is to study the group generated by the encryption functions of $\Phi$, denoted by $\Gamma(\Phi)$,  in an easier setting, i.e. the one of a long-key cipher. In particular we will focus on Feistel networks, providing a first generalisation of the results obtained in \cite{calderini} regarding translation-based ciphers.   For this purpose we also make use of the following group
\[
	\Gamma_h(\Phi)\deq \langle \veps_{h,K} \mid K \in \mathbb K \rangle ,
	\]
where all the possible round keys for round $h$ are considered.
 From this, the group
\[
	\Gamma_{\infty}(\Phi)\deq\langle \Gamma_h(\Phi)\mid 1 \leq h \leq r\rangle.
\]
can be obtained. As mentioned in Section~\ref{intro}, the group $\Gamma_{\infty}$ has been extensively studied in recent years, 
being  the closest to the one generated by the encryption function that can be successfully investigated. However it is worth stressing that $\Gamma_{\infty}(\Phi)$ may be \emph{a-priori} way larger than the actual group of the encryption functions $\Gamma(\Phi)$. \\

The \emph{imprimitivity} of such a group is one of the properties which may easily lead an attacker to a successful break of the cipher. The imprimitivity attack and its generalisation are described in the following section.

\subsection{Imprimitive action and partition-based trapdoor}\label{sect2_tb}
We recall that a permutation group $G$ acting on $V$ is called called \emph{primitive} if no non-trivial partition of $V$ is invariant under the action of $G$, 
i.e. there is no partition $\cA$ of $V$, different from the trivial partitions $\{\{v\}\mid v \in V\}$ and $\{V\}$, such that  $Ag\in \cA$ for all $A\in\cA$ and $g\in G$.
 On the other hand, if a non-trivial $G$-invariant partition $\cA$ exists, the group is called \emph{imprimitive}. Each $A \in \cA$ is called an \emph{imprimitivity block}. \\

The imprimitivity is a very undesirable property for group generated by the encryption functions of a block cipher. As Paterson \cite{CGC-cry-art-paterson1} showed, indeed, if this group is imprimitive, then it is possible to embed a trapdoor in the cipher which may allow attackers to recover crucial key-information with way less effort than a bruce force attack. 
Moreover, in~\cite{CGC-cry-art-carantisalaImp} the authors characterised the cryptographic conditions of the boolean components of a cipher which guarantee that the corresponding group $\Gamma_\infty$ is primitive. These results apply to the family of translation-based ciphers (see~\cite{CGC-cry-art-carantisalaImp}), which is large enough to contain some of the most popular encryption methods (see ~\cite{present,daemen2002design}). The conditions on the layers of the cipher which will be considered in this work are the same used in~\cite{CGC-cry-art-carantisalaImp}, or generalisation of those. 
The idea of attacking a cipher by exploiting the imprimitive action of its group has been generalized in a recent work \cite{bannier2016partition}, where the partition-based attack is introduced. The basic idea behind the attack is that, even if the group is primitive, it may exists a sequence of partitions $\mathcal A_1, \ldots, \mathcal A_r$ such that the $i$-th round function of each encryption function maps $\mathcal A_i$ into $\mathcal A_{i+1}$. It is not hard to notice that, provided that this condition is true, the cipher can be attacked using an argument similar to the one exploiting the imprimitivity. In  \cite{bannier2016partition},  the authors show an example of such attack on an SPN.\\ 

We report here some of the definitions and results presented in \cite{bannier2016partition}.

\begin{definition}
Let $\gr \in \Sym(V)$ and $\cA, \cB$ be two partitions of $V$. Let $\cA \gr$ denote the set $\{A\gr \mid A \in \cA\}$. We say that $\gr$ maps $\cA$ into $\cB$ if $\cA\gr = \cB$. Moreover, if $G$ is a permutation group we say that $G$ maps $\cA$ into $\cB$ if for all $\gr\in G$, $\gr$ maps $\cA$ into $\cB$.
\end{definition}

\begin{definition}
 A partition $\cA$ of $V$ is called \emph{linear} if there exists $U<V$ such that $$\cA=\{ U+v\mid v\in V\}.$$ We denote $\cA$ by $\cL(U)$.
\end{definition}

The following result, introduced by Harpes and Massey in \cite{harpes1997partitioning}, characterizes the possible partitions $\cA$ and $\cB$ such that the translation group $T(V)$ maps $\cA$ into $\cB$.

\begin{proposition}\label{prop:partT}
Let $\cA$ and $\cB$ be two partitions of $V$. Then $T(V)$ maps $\cA$ into $\cB$ if and only if $\cA=\cB$ and $\cA$ is a linear partition.
\end{proposition}

We report now the main result of \cite{bannier2016partition}. 

\begin{theorem}\label{th:francesi}
Let $\Phi$ be an $r$-round long-key SPN on $\mathbb M =V$. Suppose that there exist non-trivial partitions $\cA$ and $\cB$ such that for each key $K$ the encryption function $E_K$ maps $\cA$ to $\cB$. Define $\cA_1=\cA$ and $\cA_{i+1}=\cA_{i}\rho_{i}$ for $1\le i\le r$, where $\rho_i$ is the classical round function for the $i$-th round. Assume also that $\rho_1$ is the identity map. Then
\begin{itemize}
\item $\cA_{r+1}=\cB$
\item $\cA_{i}$ is a linear partition for any $1\le i\le r+1$.
\end{itemize}
\end{theorem}

In the previously shown result, Bannier et al. proved that the only partitions which propagate round-by-round are the linear ones. The next results, proved in~\cite{calderini}, shows which conditions are sufficient to avoid the linear-partition propagation  in the SPN case. The aim of this work is to provide a partial generalisation of these results in the  Feistel network case. 

\begin{proposition}\label{prop:blocchi}
Let $\gamma \in \Sym(V)$ be a parallel S-box, i.e. $\gamma=(\gamma^{(1)},...,\gamma^{(b)})$ with $\gamma^{(i)}\in\Sym(V_i)$ for all $1 \leq i \leq b$. Suppose that for all $1 \leq i \leq b$ the function $\gamma^{(i)}$ is
\begin{itemize}
\item  $2^\delta$-differentially uniform, with $\delta<m$,
\item strongly $(\delta-1)$-anti-invariant. 
\end{itemize}
Let $\cL(U)$ and $\cL(W)$ be non-trivial linear partitions of $V$. Then $\gamma$ maps $\cL(U)$ into $\cL(W)$ if and only if $U$ and $W$ are wall. Moreover $U=W$.
\end{proposition}
As a consequence, Calderini derived the following result, which guarantees immunity from the partition-based attack~\cite{calderini}.

\begin{theorem}\label{th:partcal}
Let  $\rho_1,\ldots,\rho_r\in \mathrm{Sym}(V)$ and 
let $\Phi$ be an $r$-round SPN on $\mathbb{M}=V$, where the $i$-th round applies  $\rho_i=\gamma_i\gl_i$ such that $0\rho_i=0$. Let us assume that for some $1\leq i < r$ we have
\begin{itemize}
\item $\g_i$ and $\g_{i+1}$ are parallel maps which apply $2^\delta$-differentially uniform and $(\delta-1)$-strongly anti-invariant S-boxes, for some $\delta < m$,
\item $\gl_i$ a strongly-proper diffusion layer.
\end{itemize}
Then no encryption function $E_K$ maps a non-trivial partition of $V$ into a non-trivial partition of $V$. 
\end{theorem}

\section{Results}
As previously mentioned, the aim of this work is to prove, for long-key Feistel networks, some results which are linked to those recalled in the previous section. We study the linear partition-propagation under the action of a long-key Feistel network. The results obtained may be considered as a starting-point for a complete generalisation of the the results of Sect.~\ref{sect2_tb} to Feistel networks, proved in~\cite{calderini} for translation-based ciphers.
For this purpose, let us consider a typical Feistel structure. Let us introduce a formal $2n \times 2n$ matrix which implements the Feistel structure. Such a formal matrix is defined as 
$$\bar \rho \deq \begin{pmatrix}0_n&1_n\\1_n&\rho\end{pmatrix},$$
where $0_n$  in the $n\times n$ zero matrix, $1_n$ is the $n\times n$ identity matrix and $\bar\rho$ is called  \emph{Feistel operator induced by the generating function $\rho$}, whose right action on $(x_1, x_2)\in V\times V$ is given by
$$
(x_1,x_2)\bar \rho=(x_1,x_2)\begin{pmatrix}0_n&1_n\\1_n&\rho\end{pmatrix}\deq (x_2, x_1 + x_2\rho).
$$
Note that $\bar\rho$ has the inverse matrix

$$\bar \rho^{-1} \deq \begin{pmatrix}\rho&1_n\\1_n&0_n\end{pmatrix}.$$

\noindent Let us define
$$
\begin{array}{rcl}
\sigma_{(h,k)}: V\times V & \rightarrow & V\times V\\
(x_1,x_2) & \mapsto  &(x_1+k,x_2+h),
\end{array}
$$
and
\[
T(V\times V)\deq\{\sigma_{(h,k)} \mid (h,k)\in V\times V \}.
\]
Let $\Phi$ be an $r$-round long-key Feistel network acting on $V\times V$,  having the following $i$-th round function
\[
\varepsilon_{i,K}=\bar \rho_i \sigma_{(0,k_i)},
\]
where $\bar\rho_i$ is the $i$-th Feistel operator induced by $\rho_i$ and $k_i$ is the $i$-th round key.
In this setting
\begin{equation}\label{gamma_LK}
\Gamma(\Phi)
=\langle \bar\rho_1\sigma_{(0,k_1)}\cdots\bar\rho_r\sigma_{(0,k_r)} \mid 
(k_1,\ldots,k_r)\in V^r\rangle.
\end{equation}
\begin{lemma}\label{lm:tv}
If $\Phi$ is a long-key Feistel network as above, then
$$\langle \bar \rho_1 \bar \rho_2\cdots\bar \rho_r, T(V\times V)\rangle<\Gamma(\Phi).$$
In particular $T(V\times V)<\Gamma(\Phi)$.
\end{lemma}
\begin{proof}
In order prove that $\bar \rho_1 \bar \rho_2\cdots\bar \rho_r\in \Gamma(\Phi)$, it is sufficient to consider the key $(k_1,\ldots,k_r)=(0,\ldots,0)$. Moreover, considering the key $(0,\ldots,0,k_r)$,  we obtain
$\bar \rho_1 \bar \rho_2\cdots\bar \rho_r\gs_{(0,k_r)}\in \Gamma(\Phi)$, and so $\gs_{(0,k_r)}\in \Gamma(\Phi)$ for all $k_r\in V$. 
Finally,
\[
(x_1,x_2) \begin{pmatrix}0_n&1_n\\1_n&\rho_i\end{pmatrix}\sigma_{(0,k)}=(x_2, x_1+x_2\rho + k)=(x_1,x_2) \sigma_{(k,0)} \begin{pmatrix}0_n&1_n\\1_n&\rho_i\end{pmatrix},
\]
for any $1 \leq i \leq r$, $k\in V$ and $(x_1,x_2) \in V\times V$, so we have
$\bar \rho_i \gs_{(0,k)}=\gs_{(k,0)}\bar \rho_i$, for any $1 \leq i \leq r$ and $k\in V$. Therefore  
\[
\bar \rho_1\gs_{(0,k)} \bar \rho_2\cdots\bar \rho_r=\gs_{(k,0)}\bar \rho_1 \bar \rho_2\cdots\bar \rho_r,
\]
for any $k\in V$, and so
 $\gs_{(h,k)}\in \Gamma(\Phi)$ for any $(h,k)\in V\times V$. The claim then derives by the fact that $\gs_{(h,0)}\gs_{(0,k)}=\gs_{(h,k)}$.
\end{proof}

In the following theorem we study which partitions can propagate in a long-key Feistel network.

\begin{theorem}\label{th:partizioni}

Let $\Phi$ be an $r$-round long-key Feistel network on $\mathbb M =V\times V$.  Suppose that there exist non-trivial partitions $\cA$ and $\cB$ such that for each key $K$ the encryption function $E_K$ maps $\cA$ to $\cB$. Define $\cA_1=\cA$ and  $\cA_{i+1}=\cA_{i}\bar\rho_{i}$, for $1\le i\le r-1$, where $\bar\rho_i$ is the Feistel operator induced by the generating function $\rho_i$ for the $i$-th round. Then,
\begin{itemize}
\item $\cA_{r+1}=\cB$
\item  $\cA_{i}$ is a linear partition, for any $2\le i\le r$.
\end{itemize}
Moreover, if $\cA=\cB$, i.e. $\Gamma(\Phi)$ acts imprimitively, then $\cA$ is a linear partition.
\end{theorem}
\begin{proof}

For any $1\le i\le r-1$ and any $(x,y)\in V\times V$, we have

\begin{align*}
(x,y)\bar\rho_i\gs_{(h_i,k_i)}\bar\rho_{i+1}\gs_{(0,k_{i+1})} &= (x+y\rho_i+k_i,y+h_i+(x+y\rho_i+k_i)\rho_{i+1}+k_{i+1})\\
&= (x,y)\bar\rho_i\gs_{(0,k_i)}\bar\rho_{i+1}\gs_{(0,h_i+k_{i+1})},
\end{align*}
and so
\[
\bar\rho_i\gs_{(h_i,k_i)}\bar\rho_{i+1}\gs_{(0,k_{i+1})} =\bar\rho_i\gs_{(0,k_i)}\bar\rho_{i+1}\gs_{(0,h_i+k_{i+1})}
\] for any possible choice of $h_i, k_i, k_{i+1}\in V$.
This implies that for any possible choice of $h_{1},k_{1},\ldots,h_{r-1},k_{r-1}, k_{r} \in V$ the map
\[
E=\bar\rho_1 \gs_{(h_{1},k_{1})}\bar\rho_{2}\cdots \bar\rho_{r-1} \gs_{(h_{r-1},k_{r-1})}\bar\rho_{r} \gs_{(0,k_{r})}
\]
is an element of $\Gamma(\Phi)$ as defined in Eq.\eqref{gamma_LK}.
Therefore, for any $1\le i\le r-1$, each possible map $\sigma_{(h_i,k_i)} \in T(V\times V)$ appears between $\bar\rho_i$ and $\bar\rho_{i+1}$, and so we have a similar scenario of Theorem \ref{th:francesi} (\cite[Theorem 3.4]{bannier2016partition}). Hence, proceeding as in Theorem \ref{th:francesi}, the desired claim follows.
Moreover, if $\cA=\cB$  the group $\Gamma(\Phi)$ acts imprimitively on $V\times V$, since all its the generators map the partition $\cA$ into itself. From Lemma \ref{lm:tv},  $\cA$ is a block system also for $T(V\times V)$, and so, by Proposition~\ref{prop:partT}, $\cA$ is linear.
\end{proof}

\begin{remark}
Note that we have defined the action of a round function of $\Phi$ on $V \times V$ in a such way that the corresponding round key acts on the right side of the message after applying the generating function $\rho$ on the right factor of $V\times V$. In some real-case scenarios, however, it may be possible that $\rho$ acts  after the action of the corresponding round key. If this is the case,  the $i$-th round function is defined in the following way:
$$
\varepsilon_{i,K}=\sigma_{(0,k_i)}\bar \rho_i \sigma_{(k_i,0)}.
$$
Indeed 
\begin{align*}
(x_1,x_2)\varepsilon_{i,K}& =(x_1,x_2) \sigma_{(0,k_i)}\bar \rho_i \sigma_{(k_i,0)}\\
& =(x_1,x_2+k_i) \bar \rho_i \sigma_{(k_i,0)}\\
&=(x_2+k_i,x_1+(x_2+k_i)\rho_i)  \sigma_{(k_i,0)}\\
&=(x_2,x_1+(x_2+k_i)\rho_i).
\end{align*}
In this setting we have that the group of the cipher with a long-key key-schedule is
\[
G\deq \langle\sigma_{(0,k_1)} \bar\rho_1\sigma_{(k_1,k_2)}\cdots\sigma_{(k_{r-1},k_{r})}\bar\rho_r \sigma_{(k_r,0)} \mid 
(k_1,\ldots,k_r)\in V^r\rangle,
\]
and so we have $\bar \rho_1 \bar \rho_2\cdots\bar \rho_r\in G$.  We cannot prove that $G$ contains $T(V\times V)$ as well.
Note that, as observed in the proof of Theorem \ref{th:partizioni}, any function of the type 
$$\sigma_{(0,k_1)} \bar\rho_1\sigma_{(k_1,k_2)}\cdots\sigma_{(k_{r-1},k_{r})}\bar\rho_r \sigma_{(k_r,0)}$$ can be represented as a function of type 
$$\sigma_{(0,k_1)} \bar\rho_1\sigma_{(0,k_2)}\bar\rho_2\sigma_{(0,k_1+k_3)}\cdots\sigma_{(0,k_{r-2}+k_{r})}\bar\rho_r \sigma_{(k_r,k_{r-1})},$$ which is an element of $\Gamma(\Phi)$, recalling that $\Phi$ represents the cipher where the key addition is applied after the generating function.
Thus, studying the properties of $\Gamma(\Phi)$ gives also important informations on $G$, e.g. if $\Gamma(\Phi)$ acts imprimitively, then so does $G$. More in general, partitions for $\Gamma(\Phi)$ are also partition for $G$.

\end{remark}
 In what follows, we aim at studying algebraic conditions which need to be satisfied by some partitions to prevent the partition-based attack. In particular,
 we classify a family of block systems which, in the case of Feistel networks, cannot be exploited for partition-based cryptanalysis. It is important to point out that the considered set of block systems contains the most used type of partitions for cryptanalysis. 
In order to do so, we need to study the subgroups of the direct product $(V \times V, +)$. We make use of the following result, due to 
 Goursat~\cite[Sections 11--12]{goursat}, which characterises the
subgroups of the direct product  of two groups in terms of
suitable sections of the direct factors (see
also~\cite{Petrillo}).  

\begin{theorem}[Goursat's Lemma \cite{goursat}]
  Let $G_1$  and $G_2$ be two  groups. There
  exists  a bijection between  
  \begin{enumerate}
  \item 
    the set  of all subgroups  of the 
    direct  product  $G_1\times   G_2$,  and  
  \item 
    the  set   of  all  triples
    $(A/B,C/D,\psi )$, where 
    \begin{itemize}
    \item $A$ is a subgroup of $G_{1}$,
    \item $C$ is a subgroup of $G_{2}$,
    \item $B$ is a normal subgroup of $A$,
    \item $D$ is a normal subgroup of $C$, and
    \item $\psi: A/B\to C/D$ is a group isomorphism.
    \end{itemize}
  \end{enumerate}

\noindent In this bijection, each subgroup of $G_1\times G_2$ can be uniquely
  written as
  \begin{equation*}
    U_{\psi}= \{
      (a,c) \in A \times C 
      :
      (a + B) \psi =c + D
      \}.
  \end{equation*}
\end{theorem}

\noindent Note that the isomorphism $\psi: A/B\to C/D$ is induced by a homomorphism $\varphi: A \to C$ such that $(a+B)\psi=a\varphi + D$ for any $a\in A$, and $B\varphi\leq D$. Such homomorphism is not unique. 
\begin{lemma}[\cite{GOST_ric}]\label{lemma:psiforphi}
  In the above notation, given any homomorphism $\phi$ inducing $\psi$, we have 
 \begin{equation}\label{eq:upsi}
    U_{\psi}
    =
    \{
      (a, a \varphi + d)
      :
      a \in A, d \in D
      \}.
  \end{equation}
\end{lemma}
\begin{proof}
Note first that the  right-hand side of~\eqref{eq:upsi} is contained
  in $U_{\psi}$, since for $a \in A$ and $d \in D$ we have
$
    (a + B) \psi = a\varphi + D =  a\varphi+ d
    + D, 
$
  that is, $(a, a\varphi+ d) \in U_{\psi}$. 
  Moreover $U_{\psi}$ is contained in the right-hand side
  of~\eqref{eq:upsi}. Indeed, if $(a, c) \in U_{\psi}$ we have
$
    a \varphi + D
    =
    (a + B) \psi
    =
    c + D,
$
  so that $c = a \phi + d$ for some $d \in D$.
\end{proof}
In the following result we consider two subgroups of $V \times V$ such that the first is mapped into the second by a Feistel operator. We highlight some condition that such subgroups have to satisfy. We will use the conditions derived from the next lemma also in Theorem~\ref{th:propbloc} and in Theorem~\ref{thm_blocchi}.

\begin{lemma}\label{lm:propsub}
Let  $\rho\in \mathrm{Sym}(V)$ be such that $0\rho=0$ and let $\bar\rho$ be the corresponding Feistel operator. Suppose that there exist two subgroups   $\mathcal{U}_1=
    \{
      (a_1, a_1 \varphi_1 + d_1)
      \mid
      a_1 \in A_1, d \in D_1
      \}$ and $\mathcal{U}_2=
    \{
      (a_2, a_2 \varphi_2 + d_2)
      \mid
      a_2 \in A_2, d_2 \in D_2
      \}$, $\mathcal{U}_1, \mathcal{U}_1 \leq V \times V$, where $A_i$, $D_i$ and $\varphi_i$  are as in Lemma \ref{lemma:psiforphi}, such that 
$$
\mathcal{U}_{1}\bar{\rho}=\mathcal{U}_{2}.
$$
The following properties hold true:
\begin{enumerate}
\item $\Ker\,\varphi_1\leq D_{2}$;
\item $D_2\leq A_1$;
\item $A_2=A_1\varphi_1+D_1$;
\item $D_2\varphi_1\leq D_1$.
\end{enumerate}
Moreover, 
\begin{itemize}
\item[(i)] if $D_1=\{0\}$ and  $D_2=\{0\}$, then $\rho$ is linear on $A_2$;
\item[(ii)] if $\mathcal{U}_1=A_1\times D_1$ and $\mathcal{U}_2=A_2\times D_2$, then $D_1=A_2$ and $D_2=A_1$.
\end{itemize}
\end{lemma}
\begin{proof}
By assumption, for any $a_1\in A_1$ and $d_1\in D_1$ there exist $x_{2}\in A_{2}$ and $y_{2}\in D_{2}$ such that 
\[
(a_1,a_1\varphi_1+d_1)
\begin{pmatrix}0_n&1_n\\1_n&\rho\end{pmatrix}=
(x_{2},x_{2}\varphi_{2}+y_{2}),
\]
that is
\begin{equation*}\label{eq:rho1}
(a_1\varphi_1+d_1,a_1+(a_1\varphi_1+d_1)\rho)
=
(x_{2},x_{2}\varphi_{2}+y_{2}).
\end{equation*}
From this we derive $a_1\varphi_1+d_1=x_2$ and so $A_1\varphi_1+D_1\leq A_2$. Moreover, since $\varphi_1$ is a homomorphism, we have
\[
a_1+(a_1\varphi_1+d_1)\rho=a_1\varphi_1\varphi_{2}+d_1\varphi_{2}+y_{2},
\]
hence, considering $d_1=0$ and $a_1\in \Ker\, \varphi_1$, we obtain $\Ker\,\varphi_1\leq D_{2}$.\\
Similarly, from $\mathcal{U}_{2}\bar{\rho}^{-1}=\mathcal{U}_1$, we obtain that for any $a_{2}\in A_{2}$ and $d_{2}\in D_{2}$ there exist $x_{1}\in A_{1}$ and $y_{1}\in D_{1}$ such that 
\[
(a_{2},a_{2}\varphi_{2}+d_{2})
\begin{pmatrix}\rho&1_n\\1_n&0_n\end{pmatrix}=
(x_{1},x_{1}\varphi_{1}+y_{1}),
\]
that is
\begin{equation}\label{eq:rhoinv1}
(a_{2}\rho +a_{2}\varphi_{2}+d_{2},a_{2})
=
(x_{1},x_{1}\varphi_{1}+y_{1}).
\end{equation}
From this it follows $a_{2}\rho +a_{2}\varphi_{2}+d_{2} \in A_1$, and considering $a_2=0$ we have $d_{2} \in A_1$ for any $d_2\in D_2$, and so $D_2 \leq A_1$. Moreover, since  $a_2=x_{1}\varphi_{1}+y_{1}$, we have $A_2\leq A_1\varphi_1+D_1$, and so $A_1\varphi_1+D_1=A_2$.
By Eq.\eqref{eq:rhoinv1}, we also obtain
\[
a_{2}+(a_{2}\rho+a_{2}\varphi_{2}+d_{2})\varphi_{1}=d_1,
\]
from which it follows that $D_{2}\varphi_1\leq D_1$, considering $a_{2}=0$.\\
\noindent If $D_1=D_2=\{0\}$, since $\Ker\,\varphi_1\leq D_{2}$, we have that $\varphi_1$ is an isomorphism. Then from \eqref{eq:rhoinv1} we obtain
$$
a_2\rho=a_{2}\varphi_{2}+a_2\varphi_1^{-1},
$$
for any $a_2\in A_2$. The last equation implies that $\rho$ acts linearly over $A_2$.

\noindent If $\mathcal{U}_1=A_1\times D_1$ and $\mathcal{U}_2=A_2\times D_2$, then  $A_1\varphi_1\leq D_1$ and $A_2\varphi_2\leq D_2$. So $A_2=D_1$, since $A_2 = A_1 \varphi_1 + D_1$. Finally, since $(a_1,0)\bar\rho=(0,a_1)$ we obtain that $A_1\leq D_2$, and so $A_1=D_2$.
\end{proof}
In the following theorem we show that the study of the partition propagation after two rounds of a Feistel network can be reduced to the study of the partition propagation in a round of the corresponding SPN.
A similar argument is used to provide a reduction from the primitivity of the group generated by a Feistel network to the one of the related SPN  \cite{wave}.
\begin{theorem}\label{th:propbloc}
Let  $\rho_1,\rho_2\in \mathrm{Sym}(V)\setminus \mathrm{AGL}(V)$ and let $\bar\rho_1$ and $\bar\rho_2$ be the corresponding Feistel operators. Suppose that there exist two non-trivial and proper subgroups $\mathcal{U}_1$ and $\mathcal{U}_2$  of $V\times V$ such that
\begin{enumerate}
\item  for each $(v_{1},w_{1})\in V\times V$ there exists $(v_{2},w_{2}) \in V\times V$ such that
\[
(\mathcal{U}_{1}+(v_{1},w_{1}))\bar{\rho}_1=\mathcal{U}_{2}+(v_{2},w_{2}),
\]
 \item  for each  $(v_{2 },w_{2})\in V\times V$ there exists $(v_{1},w_{1}) \in V\times V$ such that
\[
(\mathcal{U}_{2}+(v_{2},w_{2}))\bar{\rho}_2=\mathcal{U}_{1}+(v_{1},w_{1}).
\]
\end{enumerate}
Then there exist $U_1$ and $W_1$ non-trivial and proper subgroups of $V$ such that 
 for each $v\in V$ there exists $w\in V$ such that
\[
(U_1+v)\rho_1 = W_1+w.
\]
Analogously, then there exist $U_2$ and $W_2$ non-trivial and proper subgroups of $V$ such that 
 for each $v\in V$ there exists $w\in V$ such that
\[
(U_2+v)\rho_2 = W_2+w.
\]

\end{theorem}
\begin{proof}
By Lemma~\ref{lemma:psiforphi} we have
$$
\mathcal{U}_i=\{(a_i,a_i\varphi_i+d_i) \mid a_i \in A_i \text{ and } d_i\in D_i\}
$$
for each $i=1,2$. What follows from now on holds for both $i=1$ and $i=2$, where if $i=2$ we consider $i+1$ as $(i+1) \mod 2=1$. We can assume without loss of generality that $0\rho_1=0\rho_2=0$. Using assumptions 1. and 2., applying Lemma \ref{lm:propsub} we obtain $A_{i+1}= A_{i}\varphi_{i}+D_i$, $D_{i+1}\varphi_i\leq D_i$ and $\Ker\,\varphi_i\leq D_{i+1}$. Since $(v,w)\begin{pmatrix}0_n&1_n\\1_n&\rho_i\end{pmatrix}=(w,v+w\rho_i)$ for each $(v,w)\in V\times V$, in assumptions 1. and 2. we can assume $v_{i+1}=w_{i}$ and $w_{i+1}=v_{i}+w_{i}\rho_i$. Therefore, in the general case, for any $a_i\in A_i$, $d_i\in D_i$ and $(v_i,w_i)\in V\times V$ there exist $x_{i+1}\in A_{i+1}$ and $y_{i+1}\in D_{i+1}$ such that
\[
(a_i+v_i,a_i\varphi_i+d_i+w_i)
\begin{pmatrix}0_n&1_n\\1_n&\rho_i\end{pmatrix}=
(x_{i+1}+w_i,x_{i+1}\varphi_{i+1}+y_{i+1}+v_i+w_i\rho_i)
\]
that is, since the maps $\varphi_i$ are homomorphisms,
\[
a_i+(a_i\varphi_i+d_i+w_i)\rho_i=a_i\varphi_i\varphi_{i+1}+d_i\varphi_{i+1}+y_{i+1}+w_i\rho_i.
\]
Hence, considering $a_i=0$, it follows that
\[
(D_i+w_i)\rho_i\leq D_{i+1}+w_i\rho_i,
\]
yielding $|D_i|\leq |D_{i+1}|$ for $i \in \{1,2\}$, therefore $|D_1| = |D_2|$. Consequently, for $i \in \{1,2\}$ we obtain 
\begin{equation}\label{eq:blocchi}
(D_i+w_i)\rho_i =D_{i+1}+w_i\rho_i.
\end{equation}
From this the desired result follows, provided that $D_i$ and $D_{i+1}$ are both different from $(\FF_2)^n$ and both different from $\{0\}$. 
First note that $D_i = (\FF_2)^n$ if and only if $D_{i+i} = (\FF_2)^n$. Analogously  $D_i = \{0\}$ if and only if $D_{i+i} =\{0\}$.

\vspace{2mm}

\noindent$\mathbf{\left[D_i=(\FF_2)^n\right]}$ 
Since $D_1\leq A_{2}$ and $D_2\leq A_1$, we have $A_{2}=(\FF_2)^n=A_1$. Therefore $C_i=B_i=(\FF_2)^n$ for $i=\{1,2\}$, since $A_i/B_i \cong C_i/D_i$, and so $\mathcal{U}_{i}$ is trivial, a contradiction.
\vspace{2mm}

\noindent$\mathbf{\left[D_i=\{0\}\right]}$
Since $\Ker\,\varphi_1\leq D_2=\{0\}$ and $\Ker\,\varphi_2\leq D_1=\{0\}$, we have that 
$$
\psi_i=\varphi_i:A_i\xrightarrow{\cong} C_i
$$ 
is an isomorphism. Therefore
$$
A_{i}\cong C_{i}=A_{i}\varphi_{i}=A_{i+1}\cong C_{i+1}=A_{i+1}\varphi_{i+1}
$$
and in particular $|A_{i}|=|A_{i+1}|$.\\
Since $D_i=D_{i+1}=\{0\}$,  for any $a_i\in A_i$ and $(v_i,w_i)\in V\times V$ there exists $x_{{i+1}}\in A_{{i+1}}$  such that
\[
(a_i+v_i,a_i\varphi_i+w_i)
\begin{pmatrix}0_n&1_n\\1_n&\rho_i\end{pmatrix}=
(x_{{i+1}}+w_i,x_{{i+1}}\varphi_{{i+1}}+v_i+w_i\rho_i),
\]
that is
\begin{equation}\label{eq:blocchiD=0}
a_i+(a_i\varphi_i+w_i)\rho_i=a_i\varphi_i\varphi_{i+1}+w_i\rho_i.
\end{equation}
If $A_i=\{0\}$, then $C_i=\{0\}$ and so $\mathcal{U}_i$ is trivial,  a contradiction. Other\-wise, if $A_i=(\FF_2)^n$, then $A_i\varphi_i=A_{i+1}=(\FF_2)^n$ and $a_i\varphi_i+w_i$ is an element of $A_{i+1}=(\FF_2)^n$. Hence in Eq.~\eqref{eq:blocchiD=0} we can consider $w_i=0$, obtaining
\[
a_i+(a_i\varphi_i)\rho_i=a_i\varphi_i\varphi_{i+1}.
\]
Since the function $x \mapsto x + x\varphi_i\varphi_{i+1}$ is linear, we proved that $\rho_i \in \AGL(V )$, which is a contradiction since by hypothesis we are assuming $\rho_i \in \Sym(V) \setminus\AGL(V )$. \\
If $A_i< (\FF_2)^n$, for $i \in \{1,2\}$ we obtain
\[
(a_i\varphi_i+w_i)\rho_i=a_i+a_i\varphi_i\varphi_{i+1}+w_i\rho_i.
\]
Since $a_i$ and $a_i\varphi_i\varphi_{i+1}$ is contained in $A_i$  and  $a_i\varphi_i$ is an element of $A_{i+1}$ for each $a_i\in A_i$, and $|A_i|=|A_{i+1}|$, then by Eq.~\eqref{eq:blocchiD=0} we obtain
\begin{equation}\label{eq:blocchifinale}
(A_{i+1}+w_i)\rho_i=A_i+w_i\rho_i,
\end{equation}
with $A_i, A_{i+1}<(\FF_2)^n$. 
This concludes the proof: indeed, if $D_1$ and $D_2$ are both proper and non-trivial subgroups of $V$, the claim follows from Eq.~\eqref{eq:blocchi}. Otherwise, the claim follows from Eq.~\eqref{eq:blocchifinale}.
\end{proof}

The following result examines the converse implication of Theorem~\ref{th:propbloc}.

\begin{theorem}\label{th:converse}
Let  $\rho\in \mathrm{Sym}(V)$ and let $\bar\rho$   be the corresponding Feistel operator. 
If there exist $\cL(U_1)$ and $\cL(U_2)$ non-trivial linear partitions of $V$ such that $\cL(U_1)\rho=\cL(U_2)$, then there exist two non-trivial linear partitions of $V\times V$ such that $\cL(\cU_1)\bar\rho=\cL(\cU_2)$.
\end{theorem}
\begin{proof}
Since $\cL(U_1)\rho=\cL(U_2)$, for each $v\in V$ there exists $w\in V$ such that
\begin{equation}\label{eq:partU}
(U_1+v)\rho = U_2+w.
\end{equation}
Notice that in Eq.~\eqref{eq:partU} we can choose $w=v\rho$.

Let $\mathcal{U}_1\deq\{(u,u')\in U_2\times U_1\}$ and $\mathcal{U}_2\deq\{(u,u')\in U_1\times U_2\}$. Since $U_{1},\, U_2$ are not trivial, then also $\mathcal{U}_1$ and $\mathcal{U}_2$ are non-trivial. 
Let $(u,u')\in \mathcal{U}_1$ and $(v,v')\in V\times V$, then we have
\[
(u+v,u'+v')\bar\rho=(u+v,u'+v')\begin{pmatrix}0_n&1_n\\1_n&\rho\end{pmatrix}=\bigl(u'+v',u+v+(u'+v')\rho\bigr).
\]
By Eq.~\eqref{eq:partU}, there exists $u''\in U_2$ such that $(u'+v')\rho=u''+v'\rho$ and so for each $(u,u')\in \mathcal{U}_1$ and $(v,v')\in V\times V$ we obtain 
\[
(u+v,u'+v')\bar\rho=(u'+v',(u+u'')+v+v'\rho)\in (U_1+v')\times(U_2+v+v'\rho).
\]
Since $(U_1+v')\times(U_2+v+v'\rho)=\mathcal{U}_2+(v',v+v'\rho)$, we have
\begin{equation*}
\bigl(\mathcal{U}_1+(v,v')\bigr)\bar\rho\subseteq \mathcal{U}_2+(v',v+v'\rho)
\end{equation*}
and so
\[
\bigl(\mathcal{U}_1+(v,v')\bigr)\bar\rho= \mathcal{U}_2+(v',v+v'\rho)
\]
since $|\mathcal{U}_1|=|(\mathcal{U}_1+(v,v'))\bar\rho|=| \mathcal{U}_2+(v',v+v'\rho)|=|\mathcal{U}_{2}|$.
\end{proof}
\begin{remark}
Notice that if $U_1=U_2$ then $\cU_1=\cU_2$. In this case, Theorem \ref{th:converse} provides the converse of Theorem 4.5 proved in \cite{wave}. In other words, the primitivity of the group $\langle \, \rho, T(V)\rangle$ is a necessary and sufficient condition for the primitivity of the group generated by the round functions of the Feistel network acting on $V\times V$ and having $\rho$ as generating function for each round.
\end{remark}

As announced, we  provide a partial generalisation of Theorem~\ref{th:partcal} in the Feistel network case. In particular, we show some types of block systems which are not usable for the purpose of the partition-based cryptanalysis. More precisely, we show that if a Feistel network has a sequence of non-trivial linear partitions which propagate from the first round to the last one, then such partitions cannot be of the type specified in the following theorem. In other words, we are studying the propagation of linear partitions under the action of $r$ rounds, where each possible key can be chosen, i.e. under the action of a long-key Feistel network. The considered Feistel network has a generating function which is the composition of a parallel S-box followed by a diffusion layer, i.e. an SPN-like generating function. The same notation of Lemma~\ref{lemma:psiforphi} is used in the following result.

\begin{theorem}\label{thm_blocchi}
Let  $\rho_1,\ldots,\rho_r\in \mathrm{Sym}(V)$ and 
let $\Phi$ be the $r$-round Feistel network where the $i$-th round applies the Feistel operator $\bar\rho_i$ induced by $\rho_i$. Let us assume that $0\rho_i=0$ and 
$\rho_i=\gamma_i\gl_i$, where
\begin{enumerate}
\item[a)] $\g_i$ is a parallel map which applies $2^\delta$-differentially uniform and $(\delta-1)$-strongly anti-invariant S-boxes, for some $\delta < s$, where $s$ denotes the dimension of each brick,
\item[b)] $\gl_i$ a linear strongly-proper diffusion layer.
\end{enumerate}
Suppose that there exists a sequence of $r+1$ non-trivial linear partitions \break $\cL(\cU_1),\ldots,\cL(\cU_{r+1})$, where $\cU_i$ is a proper and non-trivial subgroup of $V \times V$ and $\cL(\cU_i)\bar{\rho}_i=\cL(\cU_{i+1})$ for all $1 \leq i \leq r$.
Then, none of the following condition is satisfied:
\begin{enumerate}
\item there exists $1 \leq i \leq r-1$ such that $\cL(\cU_{i+1})\bar{\rho}_{i+1}=\cL(\cU_{i})$,
\item there exists $1 \leq i \leq r-1$ such that $\cU_{i}=A_i\times D_i$, $\cU_{i+1}=A_{i+1}\times D_{i+1}$ and $\cU_{i+2}=A_{i+2}\times D_{i+2}$,
\item there exists $1 \leq i \leq r$ such that $D_i=\{0\}$ and $D_{i+1}=\{0\}$,
\item there  exists $1 \leq i \leq r$ such that $A_i = \{0\}$ and $A_{i+1}=\{0\}$.
\end{enumerate}
\end{theorem}
\begin{proof}
We proceed in each case by contradiction. 
\begin{enumerate}
\item Let $1 \leq i \leq r-1$  such that $\cL(\cU_{i+1})\bar{\rho}_{i+1}=\cL(\cU_{i})$. Then, by Theorem \ref{th:propbloc},  there exist $U_i$ and $U_{i+1}$ subgroups of $V$ such that $\cL(U_i)\rho_i=\cL(U_{i+1})$ and $\cL(U_{i+1})\rho_{i+1}=\cL(U_{i})$.
Then, by Proposition \ref{prop:blocchi},  $U_i$ and $U_{i+1}$ are walls and $U_{i+1}=U_i\gl_i$, which contradicts the fact that $\gl_i$ is strongly proper.
\item Let $1 \leq i \leq r-1$ such that $\cU_{i}=A_i\times D_i$, $\cU_{i+1}=A_{i+1}\times D_{i+1}$ and $\cU_{i+2}=A_{i+2}\times D_{i+2}$. By Lemma \ref{lm:propsub}, $\cU_{i}=A_i\times A_{i+1}$, $\cU_{i+1}=A_{i+1}\times A_i$ and $\cU_{i+2}=A_i\times A_{i+1}$. This contradicts what previously proved.  
\item Let $1 \leq i \leq r$ such that $D_i=D_{i+1}=\{0\}$, which implies $\cU_i=\{(a_i,a_i\varphi_i)\,:\, a_i \in A_i\}$ and $\cU_{i+1}=\{(a_{i+1},a_{i+1}\varphi_{i+1})\,:\, a_{i+1} \in A_{i+1}\}$.  If $A_i=\{0\}$, then $\cU_i$ is trivial. Since $D_i=D_{i+1}=\{0\}$ and $\Ker\,\varphi_{i}\leq D_{i+1}$, we have that $\varphi_i$ is an isomorphism over $A_i$ and $A_i\varphi_i=A_{i+1}$. Moreover, by Lemma \ref{lm:propsub},  $\rho_i$ is linear over $A_{i+1}$.
If $A_i= (\FF_2)^n$, then $\g_i$ is linear on $V$, which contradicts the fact that $\gamma_i$ satisfies the conditions in a).
Suppose now $A_i<(\FF_2)^n$. As in the proof of Theorem \ref{th:propbloc} we obtain that 
\[
(A_{i+1}+w_i)\rho_i=A_i+w_i\rho_i
\]
for any $w_i$ in $(\FF_2)^n$. Then $\g_i$ maps the linear partition $\cL(A_{i+1})$ into $\cL(A_{i}\gl_i^{-1})$. By Proposition \ref{prop:blocchi},  $A_{i+1}=A_{i}\gl_i^{-1}$, and $A_{i+1}$ is a wall. Since $\rho_i$ is linear over $A_{i+1}$, then $\gamma_i$ is linear over $A_{i+1}$.
If $V_j$ be a brick of the wall $A_{i+1}$, then the S-box of $\g_i$ relative to the brick $V_j$ is a linear map over $V_j$, which is a contradiction. 
\item Let $1 \leq i \leq r$ such that $A_i=A_{i+1}=\{0\}$. By Lemma~\ref{lm:propsub}, $D_{i+1} \leq A_{i} = \{0\}$, hence $\cU_{i}$ is trivial. 
\qedhere
\end{enumerate}
\end{proof}
It is worth noticing that the partition used by Paterson in his construction of a DES-like trapdoor cipher (see ~\cite[Lemma 3]{CGC-cry-art-paterson1}) is as in point 2 in the previous theorem. \\

We conclude this section by observing that it is possible to prove a result similar to Theorem~\ref{thm_blocchi} using a weaker notion of differential uniformity, defined in~\cite{CGC-cry-art-carantisalaImp}, provided a larger value of strong anti-invariance. Recalling that a map $f \in \Sym\left((\FF_2)^s\right)$ is said to be \emph{weakly $\delta$-uniform} if for each $a\in (\FF_2)^s\setminus\{0\}$ we have 
\[
\bigl|\bigl\{xf+(x+a)f \mid x\in (\FF_2)^s\bigr\}\bigr|>\frac{2^{s-1}}{\delta},
\]
the following alternative result is easily checked. Its proof is obtained reasoning as in the proof of Theorem~\ref{thm_blocchi}, since Proposition~\ref{prop:blocchi} is still valid if one assumes that the S-boxes are weakly $2^\delta$-uniform and $\delta$-strongly anti-invariant.
 
 \begin{theorem}\label{thm_blocchi_weak}
Let  $\rho_1,\ldots,\rho_r\in \mathrm{Sym}(V)$ and 
let $\Phi$ be the $r$-round Feistel network where the $i$-th round applies the Feistel operator $\bar\rho_i$ induced by $\rho_i$. Let us assume that $0\rho_i=0$ and 
$\rho_i=\gamma_i\gl_i$, where
\begin{enumerate}
\item[a)] $\g_i$ is a parallel map which applies weakly $2^\delta$-uniform and $\delta$-strongly anti-invariant S-boxes, for some $\delta < s$, 
\item[b)] $\gl_i$ a linear strongly-proper diffusion layer.
\end{enumerate}
Suppose that there exists a sequence of $r+1$ non-trivial linear partitions \break $\cL(\cU_1),\ldots,\cL(\cU_{r+1})$, where $\cU_i$ is a proper and non-trivial subgroup of $V \times V$ and $\cL(\cU_i)\bar{\rho}_i=\cL(\cU_{i+1})$ for all $1 \leq i \leq r$.
Then, none of the following condition is satisfied:
\begin{enumerate}
\item there exists $1 \leq i \leq r-1$ such that $\cL(\cU_{i+1})\bar{\rho}_{i+1}=\cL(\cU_{i})$,
\item there exists $1 \leq i \leq r-1$ such that $\cU_{i}=A_i\times D_i$, $\cU_{i+1}=A_{i+1}\times D_{i+1}$ and $\cU_{i+2}=A_{i+2}\times D_{i+2}$,
\item there exists $1 \leq i \leq r$ such that $D_i=\{0\}$ and $D_{i+1}=\{0\}$,
\item there  exists $1 \leq i \leq r$ such that $A_i = \{0\}$ and $A_{i+1}=\{0\}$.
\end{enumerate}
\end{theorem}

\section{Conclusions and open problems}
In this work, partition propagation under the action of a long-key Feistel network has been investigated, and some previous results \cite{bannier2016partition,bannier2017partition,calderini}  set in a long-key SPN scenario have been generalized. In details, we proved that  only linear partitions can propagate under the action of a long-key Feistel network. Moreover, we presented some types of block systems which are not usable for the purpose of the partition-based cryptanalysis. In other words, we showed that if in a long-key Feistel network  a sequence of non-trivial linear partitions  propagate from the first round to the last one, then such partitions cannot be of some types used in specific attacks (see e.g.~\cite{CGC-cry-art-paterson1}). 
\medskip

The problem of giving a complete generalisation of Theorem~\ref{th:partcal} of \cite{calderini} to the case of Feistel networks is still open. Moreover, the optimal result  that a block-cipher designer can achieve in terms of group-theoretical security is to obtain a cipher whose corresponding group is the larger possible. For this reason, we aim at studying which conditions imply that the group of the encryption functions of a long-key cipher is the alternating or symmetric group, both in case of SPNs and Feistel networks.
\section*{Acknowledgement}
We would like to express our gratitude to the anonymous referees for their valuable comments and suggestions.



\medskip
\medskip

\end{document}